\documentclass{amsart}

\newtheorem{theorem}{Theorem}[section]
\newtheorem{thm}[theorem]{Theorem}
\newtheorem{lem}[theorem]{Lemma}
\newtheorem{prop}[theorem]{Proposition}
\newtheorem{cor}[theorem]{Corollary}
\theoremstyle{definition}

\theoremstyle{remark}
\newtheorem{rem}[theorem]{Remark}
\newtheorem{ex}[theorem]{Example}

\newcommand{\fm}{\mathfrak m}
\newcommand{\la}{\lambda}
\newcommand{\h}{\rm{H}}
\newcommand{\Ap}{\textrm{Ap}}
\newcommand{\w}{\omega}

\begin{document}

\title[Invariants of the tangent cone]{Apery and  micro-invariants of a one dimensional Cohen-Macaulay local ring
and invariants of its tangent cone}

\author{Teresa Cortadellas }
\address{Departament d'\`{A}lgebra i Geometria, Universitat de Barcelona, Gran Via 585, 08007 Barcelona}
 \email{terecortadellas@ub.edu}
\thanks{Partially supported by MTM2007-67493}

\author{Santiago Zarzuela }
\address{Departament d'\`{A}lgebra i Geometria, Universitat de Barcelona, Gran Via 585, 08007 Barcelona}
\email{szarzuela@ub.edu}
\thanks{}

\begin{abstract}
Given a one-dimensional equicharacteristic
Cohen-Macaulay local ring $A$, Juan Elias introduced in 2001 the
set of micro-invariants of $A$ in terms of the first neighborhood
ring. On the other hand, if $A$ is a one-dimensional complete
equicharacteristic and residually rational domain,  Valentina
Barucci and Ralf Froberg defined in 2006 a new set of invariants
in terms of the Apery set of the value semigroup of $A$. We give a
new interpretation for these sets of invariants that allow to
extend their definition to any one-dimensional Cohen-Macaulay
ring. We compare these two sets of invariants with the one
 introduced by the authors for the tangent cone of a
one-dimensional Cohen-Macaulay local ring and give explicit
formulas relating them. We show that, in fact, they coincide if
and only if the tangent cone $G(A)$ is Cohen-Macaulay. Some
explicit computations will also be given.
\end{abstract}
\maketitle{}

\section{Introduction}
\label{A} Let $(A, \fm)$ be a one dimensional Cohen-Macaulay local
ring with infinite residue field and set $G(\fm):= \bigoplus
_{n\geq 0} \fm ^n / \fm ^{n+1}$ for its tangent cone. In recent
years, several new families of numerical sets have been defined in
order to study its structure and properties. We will denote by $e$
the multiplicity of the ring $A$ and by $r$ its reduction number.

The authors have observed in \cite{CZ} that if $xA$ is a minimal
reduction of $\fm$ the corresponding Noether normalization
$$ F(x):= \bigoplus_{n\geq 0} \frac{x^nA}{x^n\fm }\hookrightarrow
G(\fm):= \bigoplus_{n\geq 0} \frac{\fm^n}{\fm^{n+1}}$$ provides a
decomposition of $G(\fm)$ as a direct sum of graded cyclic
$F(x)$-modules of the form
$$G(\fm)\cong F(x)\bigoplus_{i=1}^{e-1}F(x)(-r_i)
\bigoplus_{j=1}^{f}\left(\frac{F(x)}{(x^\ast)^{t_j}F(x)}\right)(-s_j)$$
for some integers $1\leq s_1\leq \cdots \leq s_{f}$ and $r_1\leq
\cdots \leq r_{e-1}$ and where $x^\ast$ denotes the class of $x$ in
$\frac{(x)}{\fm (x)}\subseteq F(x)$. In the same paper, the above
decomposition is rewritten as
$$G(\fm)\cong
\bigoplus_{i=0}^{r}\left(F(x)(-i)\right)^{\alpha_i}
\bigoplus_{i=1}^{r-1}\bigoplus_{j=1}^{r-i}\left(\frac{F(x)}{(x^\ast)^jF(x)}(-i)\right)^{\alpha_{i,j}},
$$
with $\alpha_0=1$, $\alpha_r\neq 0$ and
$\displaystyle{\sum_{i=1}^{r}\alpha_i}=e-1$.

It turns out that the numbers $\alpha _1, \dots , \alpha _r$ are
independent of the chosen minimal reduction, while the $\alpha
_{i,j}$ depend on it.  For the purpose of this paper we call
$\{\alpha _i,\alpha_{i,j}\}$ the set of invariants of the tangent
cone (with respect to $x$).
\medskip

Let $A'$ the first neighborhood ring of $A$ and assume that $A$ is
equicharacteristic and complete. Then $A$ has a coefficient field
$K$ and a transcendental element $x$ such that $W:=K[[x]]\subset
A$ is a finite extension, $xA$ being a minimal reduction of $\fm$.
Juan Elias observed in \cite{E} that $A'/A$ is a torsion finitely
generated $W$-module and that there exist integers $a_1\leq \cdots
\leq a_{e-1}$ such that
$$ \frac{A'}{A} \cong \bigoplus_{j=1}^{e-1} \frac{W}{x^{a_j}W}\, .$$
 In fact, it may be seen that $a_j\leq r$ and that
the numbers $\{a_1,\dots, a_{e-1}\}$ are independent of the chosen
minimal reduction $xA$ and he defines this set of numbers as the
set of micro-invariants of $A$. By considering $\beta_i= \# \{j;a_j=i\} $  the above decomposition  can be rewritten as
$$ \frac{A'}{A}
\cong \bigoplus_{i=1}^{r}
\left(\frac{W}{x^{i}W}\right)^{\beta_i}.$$
 For the purpose of this paper we call $\{\beta_1,\dots ,
 \beta_r\}$ the set of micro-invariants of $A$.

\medskip

Now assume instead that $A$ is a complete equicharacteristic,
residually rational local domain with multiplicity $e$; that is,
$A$ is a subring of a formal power series ring $k[[t]]$, where
$k$ is a field, with conductor $(A:k[[t]])\neq 0$. Consider the
value semigroup $S:= v(A)=\{ v(a): 0\neq a\in A\}$ and
$\Ap(S)=\{w_0=0, w_1, \dots, w_{e-1}\}$, the Apery set of $S$ with
respect to $e$; that is, the set of the smallest elements in $S$
in each congruence class modulo $e$. An element $x$ with smallest
value $v(x)=e$ generates a minimal reduction of $A$.  A subset
$\{g_0=1, g_1,\dots ,g_{e-1}\}$ is an Apery basis of $A$ with
respect to $x$ if, for each $j$, $1\leq j \leq e-1$, the following
conditions are satisfied:

\begin{enumerate}
\item $v(g_j)=w_j$, \item if $g\in \fm^i +xA \setminus
\fm^{i+1}+xA$ and $v(g)=v(g_j)$ then $g_j\in \fm^i+xA$.
\end{enumerate}

\noindent Fixed an Apery basis $\{g_0=1, g_1,\dots ,g_{e-1}\}$
with respect to $x$ one may consider, for $1\leq j \leq e-1$, the
numbers $c_j$ as the largest $i$ such that $g_j\in \fm^i +xA$.
Observe that $c_j\leq r$. Then, if $\gamma_i=\# \{j;c_j=i\}
$, we call $\{ \gamma_1, \dots ,\gamma_r \}$ the set of Apery
invariants of $A$.
\medskip

\medskip
The main purpose of this paper is to relate these three families
of invariants by giving explicit formulas describing their
relations. The formulas are expressed in terms of colon
ideals that allow to characterize when the three families
coincide: this is precisely when the tangent cone is Cohen-Macaulay.
Moreover, we do this completely in general just assuming that the
ring $A$ is Cohen-Macaulay. For that, we first extend to any one
dimensional Cohen-Macaulay local ring the definitions of the
micro-invariants introduced by Elias and the Apery invariants.
Also, some computations are made in general when
the reduction number, the embedding dimension or the multiplicity
of $A$ are small. In the case of semigroup rings all
the computations can be done in terms of usual invariants of the
semigroup itself.

\section{Background and preliminaries}
 \label{B}
First, we set up some notation and definitions.  Let $(A,\fm)$ be
a one dimensional Cohen-Macaulay local ring with infinite residue
field, embedding dimension $b$, reduction number $r$ and
multiplicity $e$.

 \subsection{Multiplicity, embedding dimension and reduction
number}
 The length of an $A$-module $M$ will be denoted by
$\lambda (M)$ and its minimum number of generators by $\mu (M)$.
The \textit{embedding dimension} of $A$ is defined as the number
$b=\la(\fm/\fm^2)=\mu(\fm)$.

An element $x$ in $\fm^s$ is called \textit{superficial of degree
$s$} if $\fm^{n+s}=x\fm^n$ for all large $n$. Superficial elements
generate $\fm$-primary ideals, and hence are regular elements of
$A$. Being the residue field $A/\fm$ infinite, the ring $A$ has
superficial elements of degree one, and the ideals they generate
are the minimal reductions reductions of $\fm$. Let $xA$ be a
minimal reduction of $\fm$. Also, in our situation, the reduction
number of $\fm$ with respect to $xA$, that is, the minimum integer
$r$ such that $\fm^{r+1}=x\fm^r$ does not depend of the chosen
minimal reduction and it will be called the \textit{reduction
number} of $A$.

We consider $\h(n):=\mu(\fm^n) =\la(\fm^n/\fm^{n+1})$ the Hilbert
function of $\fm$ and $\h^1(n)=\displaystyle{\sum_{i=0}^n} \h(i)$
its  Hilbert-Samuel function. This is of polynomial type of degree
$1$, and the \textit{multiplicity} of $A$ is defined as the
integer $e$ such that $\h^1(n)=e(n+1)-\rho$ for all large $n$.

In the nice book by Judith D. Sally \cite{S2} dedicated to the
study of the numbers of generators of ideals in local rings, it is
proved that $\la(I/xI)=\la(A/xA)=e$ for any ideal $I$ of $A$ of
height 1. Thus, taking $I=\fm^n$ one has
$$e=\la(\fm^n/x\fm^n)=\mu(\fm^n) +\la(\fm^{n+1}/x\fm^n).$$
 Thus,
$e=\mu(\fm^n)=\mu(\fm^r)$ for $n\geq r$ and
$\mu(\fm^n)=e-\la(\fm^{n+1}/x\fm^n)<e$ for $n<r$. Also, a result
of Paul Eakin and Avinash Sathaye gives the lower bound $n+1\leq
\mu(\fm^n) $ for $n\leq r$ (an elementary proof of this bound in
the one dimensional case follows from \cite[Proposition 26]{CZ}).
In particular $r\leq e-1$ and $b=e-\la(\fm^2/x\fm) \leq e$.

In order to describe $\rho$, it is easy to see that for $n\geq r$
it is satisfied the equality
$$\h^1(n)=\mu(\fm^r)(n+1)+1+\mu(\fm)+\cdots +
\mu(\fm^{r-1})-r\mu(\fm^r),$$
 thus $$\rho=r\mu(\fm^r)-(1+\mu(\fm)+\cdots + \mu(\fm^{r-1}))=
e-1+\sum_{i=1}^{r-1} \la(\fm^{i+1}/x\fm^i).$$

\subsection{The invariants of the tangent cone}

 Let $xA$ be a minimal reduction of $\fm$ and $\alpha_i,
 \alpha_{i,j}$ the numbers defined in the introduction. The $\alpha_i$'s
 and the $\alpha_{i,j}$'s can be related in terms of lengths of colon ideals. In
order to express this fact we first define the numbers $f_{i,j}$
as

\[
f_{i,j}:=\la \left(
\frac{\fm^i\cap(\fm^{i+j+1}:x^{j})}{\fm^{i+1}}\right).
\]

\begin{rem}
\label{B3} Note that $f_{i,j}=0$ if $(i,j)\notin \{ (k,l) \mid
1\leq k \leq r-1 \textrm{ and } 1\leq l \leq r-i\}$, and also
$f_{r-1,1}=0$.
\end{rem}

Then, in \cite[Proposition 3, Proposition 7]{CZ} the following
result is proved.

\begin{lem}
\label{alpha} It holds:
\begin{enumerate}
\item for $  1\leq i \leq r-1$,
\[
\begin{array}{rl}  \alpha_i&= \la (\fm^i/(\fm^i\cap(\fm^r:x^{r-i-1})+x\fm^{i-1})) \\
                            &=\la(\fm^i/(\fm^i\cap(\fm^r:x^{r-i-1})))-\la(\fm^i/(\fm^i\cap(\fm^r:x^{r-i})))\\
                            &=\mu(\fm^i)-f_{i,r-i}-\mu(\fm^{i-1})+f_{i-1,r-i+1},
                             \\
\textrm{ and  } & \\
\alpha_r&=\la(\fm^{r}/(\fm^{r+1}+x\fm^{r-1}))=\la(\fm^{r}/x\fm^{r-1})=\mu(\fm^r)-\mu(\fm^{r-1}).
\end{array}
\]

\item
\[
 f_{k,l} =\sum_{(i,j)\in \Lambda} \alpha_{i,j}
 \]  where
$\Lambda=\{ (i,j): 1 \leq i \leq  k, \, k-i+1 \leq j \leq
  k-i+l\}$.

\item The $f_{i,r-i}$'s and so, the $\alpha_i$'s are independent
of the chosen minimal reduction $xA$ of $\fm$.
\end{enumerate}
\end{lem}

\begin{rem}
\label{known}  Some direct consequences for the tangent cone can
be immediately deduced from the above result on the structure of
$G(\fm)$ as $F(x)$-module.

For instance, the equalities
\[
0=f_{r-1,1}=\sum_{1\leq i\leq r-1} \alpha_{i,r-i}
\]
 imply that $\alpha_{i,r-i}=0$. So the $F(x)$-torsion submodule
of $G(\fm)$ has the form
\[
T(G(\fm))\cong
\bigoplus_{i=1}^{r-1}\bigoplus_{j=1}^{r-i-1}\left(\frac{F(x)}{(x^\ast)^jF(x)}(-i)\right)^{\alpha_{i,j}},
\]
which always vanishes if $r\leq 2$. Thus the tangent cone is
Cohen-Macaulay for $r$
 less or equal to $2$, as it is well known.
\end{rem}

In the next lemma we resume some characterizations in terms of
colon ideals of the Cohen-Macaulay property of the tangent cone
that will be used later on.

Given $a$ in $A$ and will denote by $a^{\ast}$ the  initial form
of $a$. That is, if $v$ is the largest integer $n$ such that $a\in
\fm^{n}$ then $a^{\ast}$ is the class of a in $\fm^{v}/\fm^{v+1}
\hookrightarrow G(\fm)$. Observe that $(x^i)^{\ast}=(x^{\ast})^i$.

\begin{lem}
\label{CM} The following conditions are equivalent:
\begin{enumerate}
\item $G(\fm)$ is a Cohen-Macaulay ring.

\item $(x^{\ast})^i$ is a regular element in $G(\fm)$ for some
(all) $i\geq 1$.

\item $\fm^n \cap x^iA = x^i\fm^{n-i}$ for all $n$.

 \item $(\fm^n :x^i)=\fm^{n-i}$ for some (all) $i\geq 1$ and
all $n$.

\item $\fm^{i}\cap (\fm^r:x^{r-i-1})=\fm^{i+1}$ for $1\leq i\leq
r-2$.

\end{enumerate}
\end{lem}
\begin{proof}
We fix $i\geq 1$. The element $(x^i)^{\ast}$ is a system of
parameters of $G(\fm)$. Hence the equivalence between (1) and (2)
 is clear. Moreover, since $x$ is regular in $A$, we have by the
result of Paolo Valabrega and Giuseppe Valla \cite[Corollary
2.7.]{VV} that $(x^i)^{\ast}$ is a regular element in $G(\fm)$ if
and only $\fm^n \cap (x^i)=x^i\fm^{n-i}$ for all $n$. Moreover, by
using the regularity of $x^i$ (or x) in $A$ this last equality is
equivalent with the equality of (4).

By \cite [Proposition 2]{CZ} the $F(x)$-torsion submodule of
$G(\fm)$ is
$$T(G(\fm))=H^0_{F(x)}(G(\fm))=(0:_{G(\fm)}
(x^{\ast})^{r-1})=\bigoplus_{i=1}^{r-1}(\fm^i\cap
(\fm^{r+1}:x^{r-i}))/\fm^{i+1}.$$ The tangent cone  $G(\fm)$ is
Cohen-Macaulay if and only if it is a free $F(x$)-module. Since
$(\fm^{r+1}:x^{r-i})=(\fm^{r}:x^{r-i-1})$ we have now the
equivalence of (5) with any of the other assertions.
\end{proof}

\begin{lem}
\label{sum-alpha} The following equality holds
$$ \sum_{i=1}^{r}i\alpha_i=\rho+\la(T(G(\fm))).$$
\end{lem}
\begin{proof}
By Lemma \ref{alpha} (1) we have that
$\sum_{i=1}^{r}i\alpha_i=r\mu(\fm^r)- (1+\mu (\fm)+\cdots +\mu
(\fm^{r-1}))+f_{1,r-1}+\cdots +f_{r-2,2}=\rho + \la(T(G(\fm)).$
\end{proof}

As a consequence of the above lemma we obtain the following
characterization for the Cohen-Macaulay property of the tangent
cone.

\begin{cor}
\label{sum-alpha-rho} $G(\fm)$ is Cohen-Macaulay if and only if
$\displaystyle{\sum_{i=1}^{r}}i\alpha_i=\rho$.
\end{cor}

\subsection{The micro-invariants of the ring}
Douglas G. Northcott defined the first neighborhood ring of $A$ as the set of
all elements, in the total ring of fractions $Q(A)$ of $A$, of the
form $\frac{b}{a}$, where $b\in \fm^s$ and $a$ is a superficial
element of degree $s$. This is a subring of $Q(A)$ containing $A$
and we will denote it by $A'$. Let $\overline A$ be the integral
closure of $A$ in $Q(A)$. We summarize in the following lemma some
of the basic facts on $A'$. For their proof we refer to the works
of Eben Matlis \cite[Chapter XII]{M} and Joseph Lipman
\cite[\S1]{L}, where this ring is studied in a more general
context.

\begin{lem}
\label{A'}
 With the notations above introduced the following hold:
\begin{enumerate}
 \item $A'= A\left[\frac{\fm}{x}\right].$
 \item $A'=\displaystyle{\bigcup_{n\geq 0} (\fm^n:_{\overline{A}} \fm^n)}=
 (\fm^r:_{\overline A}\fm^r)$.
 \item $A'$ is a finitely generated $A$-module, and hence is a semi-local, one dimensional Cohen-Macaulay ring.
 \item $x$ is a regular element of $A'$.
 \item $\fm^n A'=x^n A'$ for all
 $n$.
 \item $\fm^n=x^n A'$ for $n\geq r $.
 \item If $M$ is a finitely generated $A$-submodule of $Q(A)$ that
 contains a regular element element of $A$ then $\lambda (M/xM)=e$.
 \item $\la (A'/\fm^n A')=ne$ for all $n$ and $\la (A'/A)=\rho$.
  \end{enumerate}
\end{lem}

For any one dimensional Cohen-Macaulay local ring $(A, \fm)$ we
define the \textit{micro-invariants} of $A$ as the set of integers

\[
\beta_i=\la \left(\frac{ A+\fm^{i-1}A'}{A+\fm^iA'}\right )-\la
\left( \frac{A+\fm^{i}A'}{A+\fm^{i+1}A'}\right)
\]
 for $i\in \{1,\dots
,r\}$, and $\beta_0=1$.

\begin{lem}
\label{sum-beta} The following equalities hold
\begin{enumerate}
\item $\displaystyle{\sum_{i=1}^r} \beta_i=e-1,$

\item $\displaystyle{\sum_{i=1}^r }i\beta_i =\rho.$
\end{enumerate}
\end{lem}

\begin{proof}
For $(1)$ observe that $\beta_r=\la ((
A+\fm^{r-1}A')/(A+\fm^rA'))$ since
$A+\fm^rA'=A+\fm^r=A=A+\fm^{r+1}=A+\fm^{r+1}A'$, by Lemma \ref{A'}
$(5)$ and $(6)$. So
\[
\begin{split}
\sum_{i=1}^r \beta_i & =\la(A'/A+\fm A')=\la(A'/\fm A')-
\la(A/(A\cap \fm A') \\ &=\la(A'/\fm A')- \la(A/\fm )=e-1
\end{split}
\]
On the other hand,
\[
\sum_{i=1}^r i \beta_i  =
\sum_{i=1}^{r}\la((A+\fm^{i-1}A')/(A+\fm^{i}A'))=\la(A'/A) = \rho \] by Lemma \ref{A'} $(8)$ and
so we get $(2)$.
\end{proof}

The following result is an immediate consequence of the above
lemma
 and Corollary \ref{sum-alpha-rho}.

\begin{cor}
\label{sum-alpha=sum-beta} $G(\fm)$ is Cohen-Macaulay if and only
if $\sum_{i=1}^r i \alpha_i=\sum_{i=1}^r i \beta_i$.
\end{cor}

 Assume now that $A$ is in addition equicharacteristic and
complete. Then $A$ has a coefficient field $K$, and the extension
$W:=K[[x]]\subseteq A$ is finite, where $W$ is a discrete
valuation ring. Notice that $A$ and $A'$ are finitely generated
$W$-modules without torsion and so $W$-free modules of rank $e$,
by Lemma \ref{A'} $(7)$.

Hence $ A'/A$ is a $W$-module of torsion and there exist integers
$a_0\leq \cdots \leq a_{e-1}$ such that
$$ \frac{A'}{A} \cong \bigoplus_{i=0}^{e-1} \frac{W}{x^{a_i}W}.  $$
The ideals $x^{a_0}W,\dots, x^{a_{e-1}}W$ are the invariants of
$A$ in $A'$. Elias shows in \cite{E} that $a_0=0$ and that these
numbers do not depend on $W$ as well. In fact, the following
holds, which gives the equivalence of the set of micro-invariants
as we have just defined and the one defined by Elias in \cite{E},
in the case $A$ is equicharacteristic and complete:

\begin{lem} \label{beta}
  \cite[Proposition 1-4]{E}
  For $i\geq 1$ it holds
$\beta_i=\# \{j; a_j=i\}= $ $$\la((x^rA+\fm^{r+i-1})/(x^rA+\fm^{r+i}))-\la((x^rA+\fm^{r+i})/(x^rA+\fm^{r+i+1})
.$$
\end{lem}

\begin{proof}
The first equality follows from the definition of the $\beta_i$'s,
the equalities $\fm^iA'=x^iA'$ for all $i$ (Lemma \ref{A'} $(5)$)
and from the fact that
$$\# \{j; a_j=i\} =\la((A+x^{i-1}A')/(A+x^i
A'))-\la((A+x^{i}A')/(A+x^{i+1}A')).$$

For the second equality one uses that $x$ is a regular element of
$A'$ and that $\fm^i=\fm^iA'=x^iA'$ for $i\geq r$, by Lemma
\ref{A'} $(6)$. Thus,
$$\begin{array}{l}(A+\fm^iA')/(A+\fm^{i+1}A')\cong

x^r(A+\fm^iA')/x^r(A+\fm^{i+1}A')=\\
(x^rA+\fm^{r+i}A')/(A+\fm^{r+i+1}A')=
(x^rA+\fm^{r+i})/(x^rA+\fm^{r+i+1}).\end{array}$$
\end{proof}

Observe that, as a consequence of this lemma, the above
decomposition of $A'/A$ can also be written as $$ \frac{A'}{A}
\cong \bigoplus_{i=1}^{r}
\left(\frac{W}{x^{i}W}\right)^{\beta_i}.$$

\subsection{Apery invariants}

Let $xA$ be a minimal reduction of $\fm$ and  $\overline{\fm} :=
\fm/xA$ be the maximal ideal of $A/xA$.

We define the \textit{Apery} invariants of $A$ with respect to $x$
as the set of integers
\[
\gamma_i=\dim_k \left(
\frac{\overline{m}^i}{\overline{m}^{i+1}}\right)=\la \left(
\frac{\fm^i+xA }{\fm^{i+1}+xA} \right).
\]
for $i\leq r$. That is, the values of the Hilbert-Samuel function
of the $0$-dimensional local ring $A/xA$.

\begin{lem}
\label{sum-gamma} The following equalities hold:
\begin{enumerate}
\item $\displaystyle{\sum_{i=1}^r} \gamma_i=e-1,$

\item $\displaystyle{\sum_{i=1}^r }i\gamma_i = \rho-
\displaystyle{\sum_{i=1}^{r-1} } \la(\fm^{i+1}\cap xA/x\fm^{i}) .$
\end{enumerate}
\end{lem}

\begin{proof}
By considering the exact sequences
$$ 0\longrightarrow (\fm^{i}+xA)/(\fm^{i+1}+xA) \longrightarrow  A/(\fm^{i+1}+xA)\longrightarrow
A/(\fm^{i}+xA)\longrightarrow 0$$ for $1\leq i \leq r$, and taking
lengths, the equality $ \sum_{i=1}^r
\gamma_i=\la(A/(\fm^{r+1}+xA))-\la(A/(\fm
+xA))=\la(A/xA)-\la(A/\fm )=e-1$ is deduced.

By using the above exact sequence also it is easily deduced that
$\sum_{i=1}^r i\gamma_i=re- \sum_{i=1}^r \la(A/\fm^i
+xA)=e-1+(r-1)e- \sum_{i=1}^{r-1} \la(A/\fm^{i+1} +xA)$. Then,
$\sum_{i=1}^r i\gamma_i =e-1+\sum_{i=1}^{r-1}
\la(\fm^{i+1}/\fm^{i+1}\cap xA) $
 follows by taking lengths in the exact sequences
$$ 0\longrightarrow \fm^{i+1}/(\fm^{i+1}\cap xA) \longrightarrow  A/xA\longrightarrow
A/(\fm^{i+1}+xA)\longrightarrow 0,$$ for $1\leq i \leq r-1$. Now,
the equality $e-1=\rho-\sum_{i=1}^{r-1} \la(\fm^{i+1}/x\fm^i)$,
gives $\sum_{i=1}^r i\gamma_i =\rho-\sum_{i=1}^{r-1}
\la(\fm^{i+1}/x\fm^i)+\sum_{i=1}^{r-1} \la(\fm^{i+1}/\fm^{i+1}\cap
xA)$. Finally, the exact sequences
$$ 0\longrightarrow \fm^{i}\cap xA/x\fm^{i} \longrightarrow  \fm^{i+1}/x\fm^i\longrightarrow
\fm^{i+1}/(\fm^i\cap xA)\longrightarrow 0,$$ for $1\leq i \leq
r-1$ transform the last equality into the sentence (2).
\end{proof}

\begin{cor}
\label{comparing-sums} It holds:
$$\sum_{i=1}^r i \gamma_i \leq \sum_{i=1}^r i\beta_i \leq \sum_{i=1}^r i\alpha_i,$$
and any (all) of the equalities occurs if and only if $G(\fm)$ is
Cohen-Macaulay.
\end{cor}

\begin{proof} Lemma \ref{sum-alpha}, Lemma \ref{sum-beta} and
Lemma \ref{sum-gamma} give the inequalities in the corollary. Also,
these lemmas and the characterization of the Cohen-Macaulay
property of the tangent cone of $A$ in terms of the
Valabrega-Valla conditions (reflected in Lemma \ref{CM}) and by
the vanishing of the torsion module $T(G(\fm))$ complete the
proof.
\end{proof}

Assume that $A$ is a complete equicharacteristic, residually
rational local domain of multiplicity $e$; that is, $A$ is a
subring of the formal power series ring $k[[t]]$ with conductor
$(A:k[[t]])\neq 0$.  Let us denote by $v$ the $t$-adic valuation.

We consider the value semigroup $S:= v(A)=\{ v(a): 0\neq a\in
A\}$. Then $x$ is an element of smallest positive value $e$. We
denote by $\Ap(S)=\{w_0=0, w_1, \dots, w_{e-1}\}$, the Apery set
of $S$ with respect to $e$; that is, the set of the smallest
elements in $S$ in each congruence class module $e$.

We call a subset $\{g_0=1,g_1\dots ,g_{e-1}\}$ of elements of $A$
an Apery basis with respect to $x$ if the following conditions are
satisfied for each $j$, $1\leq j \leq e-1$:
\begin{enumerate}
\item $v(g_j)=w_j$,
 \item $\max \{i \mid g_j \in \fm^i+xA \}=\max \{i \mid w_j \in v(\fm^i+xA) \}$.
\end{enumerate}

We shall denote by $c_j := \max \{i \mid g_j \in \fm^i+xA \}$. Observe that $c_j\leq r$. The following observation justifies why
we call these invariants, the Apery invariants.

\begin{lem}
\label{gamma} For $i\geq 1$, $ \gamma_i=\# \{ j; c_j=i \} .$
\begin{proof} Let  $\Ap(S)=\{w_0, w_1, \dots, w_{e-1}\}$, the Apery set
of $S$ and $\{g_0,g_1\dots ,g_{e-1}\}$ be an Apery basis  of $A$
with respect to $x$.

 Fixed $i$, we consider $\fm^i+xA$. If
$i\leq c_j$ then $g_j \in \fm^i+xA$ and obviously $w_j \in
Ap(v(\fm^i +xA ))$. If $i > c_j$ then, by definition of $c_j$,  $
w_j\notin v(\fm^{i+1} +xA ))$ and, since $xg_j\in \fm^{i+1} +xA$
with $v(xg_j)=w_j+e$, we have that $w_j+e \in Ap(v(\fm^i +xA ))$.
So, applying Lemma 2.1 of \cite{BF}, $\fm^i +xA$ is a free
$k[[x]]$-module of rank $e$ with basis $x^{\epsilon_{i,j}}g_j$
with $\epsilon_{i,j}\in \{0,1\}$. Thus, $\la( (\fm^i +xA)/xA)=
\# \{ j ; c_j \geq i\} $ and $\gamma_i := \la ((\fm^i
+xA)/(\fm^{i+1}+xA)) =\# \{j; c_j=i\}$.

\end{proof}
\end{lem}

We call  a subset $\{f_0=1,f_1\dots ,f_{e-1}\}$ of elements of $A$
a BF-Apery basis if the following conditions are satisfied for
each $j$, $1\leq j \leq e-1$:
\begin{enumerate}
\item $v(f_j)=w_j$,
 \item [(2')] $\max \{i \mid f_j \in \fm^i \}=\max \{i \mid w_j \in v(\fm^i)\}$.
\end{enumerate}
We shall denote by $b_j := \max \{i \mid f_j \in \fm^i \}$ and we say that $A$
satisfies the $BF$ condition with respect to
$x$ if $m^i$, for all $i\geq 0$, is generated freely by elements
of type $x^{h_{i,j}} f_j$, $0\leq j \leq e-1$, for some exponents
$h_{i,j}$.

Note that BF-Apery basis are called Apery basis by Barucci and Fr\"{o}berg in \cite{BF}. In general, as shown by Lance Bryant in his
Ph. Dissertation \cite{B}, the BF condition is not always satisfied. It is
easy to see that under the BF condition with respect to $x$, then $\gamma_i=\#
 \{j;
b_j=i\}$.

\section{Comparing invariants}
 Let $(A,\fm)$ be an one dimensional Cohen-Macaulay local ring
with infinite residue field, embedding dimension $b$, reduction
number $r$ and multiplicity $e$. Let $(x)=xA$ be a minimal
reduction of $\fm$. In this section we will compare the sets of
numbers introduced in the above section; that is
\begin{itemize}
\item $\{ \alpha_i,\, \alpha_{i,j}\}$ the invariants of the
tangent con $G(\fm)$ with respect to $x$.

\item $\{ \beta_1, \dots , \beta_r\}$ the micro-invariants of $A$.

\item $\{\gamma_1,\dots, \gamma_r\}$ the Apery invariants of $A$
with respect to $x$.
\end{itemize}

\subsection{Micro-invariants of the ring and invariants of its
tangent cone}
 Our first purpose is
to measure the difference between $\beta_i$ and $\alpha_i$ also in
terms of lengths of colon ideals. For this, we will begin by
writing the $\beta_i$'s in terms of lengths of specific colon ideals.

\begin{lem}
\label{2beta} For $1\leq i \leq r-1$, it holds
$$\beta_i=
\la((\fm^r:x^{r-i})/(\fm^r:x^{r-i-1}))
                              -\la((\fm^r:x^{r-i+1})/(\fm^r:x^{r-i})),$$
and $\beta_r=\mu(\fm^r)-\la((\fm^r :x)/\fm^r)$.
\end{lem}

\begin{proof}
By Lemma \ref{beta} we have that
$$\beta_i=\la(((x^r)+\fm^{r+i-1})/((x^r)+\fm^{r+i}))-
\la((x^r)+\fm^{r+i})/((x^r)+\fm^{r+i+1})).$$ Now, by considering
the exact sequence
$$0\rightarrow (x^r)\cap\fm^{r+i}/(x^r)\cap\fm^{r+i+1}
\rightarrow \fm^{r+i}/\fm^{r+i+1}
\rightarrow\fm^{r+i}/((x^r)\cap\fm^{r+i}+\fm^{r+i+1})\longrightarrow
0$$ and the isomorphisms
$$ \begin{array}{rl}((x^r)+\fm^{r+i})/((x^r)+\fm^{r+i+1})&\cong
\fm^{r+i}/(\fm^{r+i}\cap(x^r)+\fm^{r+i+1}),\\
\fm^{r+i}/\fm^{r+i+1}&\cong \fm^{r}/\fm^{r+1} \end{array}$$  we
obtain the equality
$$
\la(((x^r)+\fm^{r+i})/((x^r)+\fm^{r+i+1}))=\mu(\fm^r)-\la((x^r)\cap\fm^{r+i}/((x^r)\cap\fm^{r+i+1})).$$
Also, one can easily prove that $(x^r)\cap \fm^{r+i}=(x^r)\cap
x^i\fm^r=x^r(\fm^r:x^{r-i})$. From these considerations it may be
deduced that
$$\beta_i=\la((\fm^r:x^{r-i})/(\fm^r:x^{r-i-1}))
                              -\la((\fm^r:x^{r-i+1})/(\fm^r:x^{r-i}))$$ for $ 1\leq i \leq r-1$ and
that $ \beta_r=\mu(\fm^r)-\la((\fm^r :x)/\fm^r).$
\end{proof}

In the next proposition we express the difference between the
value of the micro-invariant and the invariant for an specific $i$
in terms of lengths of colon ideals.

\begin{prop}
\label{beta-alpha}
 For $1\leq i \leq r$ it holds
$$\beta_i +\la((\fm^r:x^{r-i+1})/(\fm^{i-1}+(\fm^r:x^{r-i})))=
\alpha_i +\la((\fm^r:x^{r-i})/(\fm^{i}+(\fm^r:x^{r-i-1}))).$$
\end{prop}

\begin{proof} For $1\leq i \leq r-1$, consider the exact sequences
$$
\begin{array}{ll}
0  &\rightarrow \fm^i/(\fm^i\cap(\fm^r:x^{r-i-1})) \rightarrow
(\fm^r:x^{r-i})/(\fm^r:x^{r-i-1})\\
  &\rightarrow
(\fm^r:x^{r-i})/(m^i+(\fm^r:x^{r-i-1}))\rightarrow 0
 \end{array}
$$
and
$$0\rightarrow (\fm^i\cap (\fm^r:x^{r-i-1}))/(\fm^{i+1})
\rightarrow \fm^i/\fm^{i+1} \rightarrow  \fm^i/(m^i\cap
(\fm^r:x^{r-i-1}))\rightarrow 0\,.$$ Taking lengths we get
$$\la((\fm^r:x^{r-i})/(\fm^r:x^{r-i-1}))=\la((\fm^r:x^{r-i})/(m^i+(\fm^r:x^{r-i-1})))+\mu
(\fm^i)-f_{i,r-i}.$$

Hence, by Lemma \ref{beta}, Lemma \ref{alpha} and the the above
lemma we get, for $1\leq i \leq r-1$, that
$$\beta_i - \alpha_i=
\la((\fm^r:x^{r-i})/(\fm^{i}+(\fm^r:x^{r-i-1})))
-\la((\fm^r:x^{r-i+1})/(\fm^{i-1}+(\fm^r:x^{r-i}))),$$ and
$\alpha_r
-\beta_r=\la((\fm^r:x)/\fm^r)-\mu(\fm^{r-1})=\la((\fm^r:x)/\fm^{r-1}).$
\end{proof}

\subsection{Apery invariants of the ring and invariants of its
tangent cone}

Put $G:=G(\fm)$, $F:=F(x)$ and $\overline{\fm}:=\fm/xA \subseteq
A/xA$.

\begin{prop}
\label{alpha-gamma} For $1\leq i \leq r$ it holds
$$\alpha_i + \sum_{j=1}^{r-i-1} \alpha_{i,j}= \gamma_i + \la((\fm^{i}\cap xA
+\fm^{i+1})/(x\fm^{i-1}+\fm^{i+1})).$$
\end{prop}
\begin{proof}
With the notation just introduced, we have an exact sequence of
modules
$$0\longrightarrow V\longrightarrow G/x^\ast G \longrightarrow
G(\overline{\fm}) \longrightarrow 0,$$ where

$$\begin{array}{ll}
V &= \bigoplus_{n\geq 0}(\fm^{n}\cap xA +\fm^{n+1})/(x\fm^{n-1}+\fm^{n+1}),\\
G/x^{\ast}G & = \bigoplus_{n\geq 0}\fm^{n}/(x\fm^{n-1}+\fm^{n+1}) \textrm{ and },\\
G(\overline{\fm}) & = \bigoplus_{n\geq
0}(\fm^{n}+xA)/(\fm^{n+1}+xA).
\end{array}$$
Taking the corresponding Hilbert series (which are polynomials of degree up to $r$) we get
$$ H_{G/x^{\ast}G}(z)=H_V(z)+H_{G(\overline{\fm})}(z).$$

By the definition of the $\gamma_i$'s we have that
$H_{G(\overline{\fm})}(z)=\sum_{i=0}^r \gamma_i z^i$. On the other
hand,
$$G/x^{\ast}G\cong
\bigoplus_{i=0}^{r}\left(F/x^{\ast}F(-i)\right)^{\alpha_i}
\bigoplus_{i=1}^{r-1}\bigoplus_{j=1}^{r-i-1}\left(\frac{F}{x^\ast
F}(-i)\right)^{\alpha_{i,j}}
$$
and so $H_{G/x^{\ast}G}(z)=\sum_{i=0}^r(\alpha_i +
\sum_{j=1}^{r-i-1}\alpha_{i,j})z^i.$ Now, taking coefficients in
the above equality between Hilbert series we get the statement.
\end{proof}

\begin{cor} The following equalities hold
\begin{enumerate}
\item $\alpha_1 +\sum_{j=1}^{r-2} \alpha_{1,j}=\gamma_1=\mu (\fm
)-1. $,
 \item $\alpha_2 +\sum_{j=1}^{r-3} \alpha_{2,j}=\gamma_2=\mu (\fm^2)-\mu (\fm
 )+\alpha_{1,1}$.
\end{enumerate}

\end{cor}

\subsection{Micro-invariants and Apery invariants of the ring}

For short we write
$$\nu_i:= \la((\fm^{i}\cap xA +\fm^{i+1})/(x\fm^{i-1}+\fm^{i+1}))$$
and
$$ g_i:=\la((\fm^r:x^{r-i})/(\fm^{i}+(\fm^r:x^{r-i-1}))).$$

Then, applying the previous results we obtain the following
relation between the micro-invariants of $A$ and the Apery
invariants of $A$ with respect to $x$:
\begin{cor}
\label{beta-gamma} For $1\leq i \leq r$ it holds
$$ \beta_i+\sum_{j=1}^{r-i-1}\alpha_{i,j}=\gamma_i+\nu_i+ g_i-g_{i-1}.$$
\end{cor}

\section{Cohen-Macaulay tangent cone}
 Let $(A,\fm)$ be an one dimensional Cohen-Macaulay local ring
with infinite residue field $K$, embedding dimension $b$,
reduction number $r$ and multiplicity $e$. Let $(x)$ be a minimal
reduction of $\fm$.

Let

\begin{itemize}
\item $\{ \alpha_i,\, \alpha_{i,j}\}$ the invariants of the
tangent con $G(\fm)$ with respect to $F(x)$.

\item $\{ \beta_1, \dots , \beta_r\}$ the micro-invariants of $A$.

\item $\{\gamma_1,\dots, \gamma_r\}$ the Apery invariants of $A$
with respect to $x$.
\end{itemize}

and, for short, we will write

$$\begin{array}{l}
f_i:=\la((\fm^i\cap (\fm^{r}:x^{r-i-1}))/\fm^{i+1})\\
g_i:=\la( (\fm^{r}:x^{r-i})/(\fm^{i}+(\fm^r:x^{r-i-1})))\\
\nu_i:=\la((\fm^{i}\cap xA +\fm^{i+1})/(x\fm^{i-1}+\fm^{i+1}))
\end{array}
$$

\begin{thm}
\label{GCM} Assume that the tangent cone of $A$ is Cohen-Macaulay,
then for $1\leq i \leq r$ it holds
$$0<\alpha_i=\beta_i=\gamma_i=\mu(\fm^{i})-\mu(\fm^{i-1}).$$
\end{thm}
\begin{proof}
By the results obtained in the above section
$$\begin{array}{l}
\alpha_i=\mu(\fm^{i})-\mu(\fm^{i-1})-f_i+f_{i-1},\\
\beta_i-\alpha_i=g_i-g_{i-1},\\
\alpha_i +\sum_{j=1}^{r-i-1} \alpha_{i,j}= \gamma_i +\nu_i,\\
\beta_i+\sum_{j=1}^{r-i-1}\alpha_{i,j}=\gamma_i+ g_i-g_{i-1}.
\end{array}$$
Then, Lemma \ref{CM} gives that $f_i=g_i=\nu_i=\alpha_{i,j}=0$ for
all $i,j$ if $G(\fm)$ is Cohen-Macaulay and the equalities hold.

Also, \cite[Corollary 16]{CZ} proves that
$\alpha_i=\la(\fm^i/(\fm^{i+1}+x\fm^{i-1}))>0$.
\end{proof}

\begin{thm}
\label{GCM2} The following conditions are equivalent:
\begin{enumerate}
\item $G(\fm)$ is a Cohen-Macaulay ring.

\item $\alpha_i=\beta_i$ for $i\leq r$.

\item $\alpha_i=\gamma_i$ for $i\leq r$.

\item $\beta_i=\gamma_i$ for $i\leq r$.
\end{enumerate}
\end{thm}

\begin{proof}
By Corollary \ref{comparing-sums} any of the conditions (2), (3) or (4) implies that $G(\fm)$ is Cohen-Macaulay.
Conversely, if the tangent cone $G(\fm)$ is Cohen-Macaulay, by
Theorem \ref{GCM} we have that (2), (3) and (4) hold.
\end{proof}

\begin{prop}
\label{alpha=mu beta=mu gamma=mu}
 Assume that any of the following equalities hold:

 \begin{enumerate}

\item[(1)] $\alpha_i=\mu(\fm^i)-\mu(\fm^{i-1})$ for $1\leq i
\leq r$;

\item[(2)]  $\beta_i=\mu(\fm^i)-\mu(\fm^{i-1})$ for $1\leq i
\leq r$;

\item[(3)] $\gamma_i=\mu(\fm^i)-\mu(\fm^{i-1})$ for $1\leq i
\leq r$.

\end{enumerate}

Then the tangent cone of $A$ is Cohen-Macaulay.
\end{prop}

\begin{proof}
(1) and (3) We observe that $\sum_{i=1}^r i(\mu(\fm^i)
-\mu(\fm^{i-1}))=\rho$. Then, if the equalities of (1) or (3)
occur, applying Lemma \ref{sum-beta} and Corollary
\ref{comparing-sums}, we obtain that $G(\fm)$ is Cohen-Macaulay.

\medskip

(2) We will prove, by induction on $i$, that
$\beta_j=\mu(\fm^j)-\mu(\fm^{j-1})$ for $1\leq j\leq i$ implies
the equality $(\fm^r:x^{r-i-1})=\fm^{i+1}$. For $i=1$,
$\beta_1=\mu(\fm)-1-f_1=\mu(\fm)-1$ gives $f_1=0$, and so,
$(\fm^r:x^{r-2})=\fm^2$. Assume
$\beta_j=\mu(\fm^j)-\mu(\fm^{j-1})$ for $1\leq j\leq i-1$. Then,
by induction, $(\fm^r:x^{r-j-1})=\fm^{j+1}$ for $1\leq j\leq i-1$.
In particular, $(\fm^r:x^{r-i})=\fm^{i}$ and
$(\fm^r:x^{r-i+1})=\fm^{i-1}$ which produces $f_{i-1}=0$,
$g_{i}=0$ and $g_{i-1}=0$. Hence, $\beta_i
=\mu(\fm^{i})-\mu(\fm^{i-1})-f_{i}= \mu(\fm^{i})-\mu(\fm^{i-1})$
implies that
$$f_{i}=\la((\fm^{i}\cap(\fm^{r}:x^{r-i-1}))/\fm^{i+1})=\la((\fm^{r}:x^{r-i-1})/\fm^{i})=0.$$
Thus, $\beta_i=\mu(\fm^i)-\mu(\fm^{i-1})$ for $1\leq i\leq r$
implies that $(\fm^r:x^{r-i})=\fm^{i+1}$ for $1\leq i\leq r-1$ and
so, $G(\fm)$ is Cohen-Macaulay.







\end{proof}

We can summarize the above results in the following way:

\begin{theorem}
\label{mainCM} The following conditions are equivalent:
\begin{enumerate}
\item $G(\fm)$ is Cohen-Macaulay.

\item $G(\fm) \cong K[X]\oplus (K[X](-1))^{\mu (\fm)-1}\oplus
\cdots \oplus  (K[X](-r))^{\mu(\fm^r)-\mu(\fm^{r-1})}$.

\item $H_{G(\fm/xA)}(z)=1+(\mu (\fm)-1)z+\cdots
+(\mu(\fm^r)-\mu(\fm^{r-1}))z^r$.
\end{enumerate}

And in the equicharacteristic and complete case also with
\begin{enumerate}

\item [(4)]  $A'/A \cong (K[[X]]/XK[[X]])^{\mu (\fm)-1} \oplus
\cdots \oplus  (K[[X]]/ X^rK[[X]])^{\mu(\fm^r)-\mu(\fm^{r-1})}$.
\end{enumerate}

\end{theorem}

\section{Some computations}

Let $(A,\fm)$ be an one dimensional Cohen-Macaulay local ring with
infinite residue field $K$, embedding dimension $b$, multiplicity
$e$ and reduction number $r$. Let $(x)$ be a minimal reduction of
the maximal ideal.

By the previous section, the micro-invariants of $A$, its Apery
numbers, and the invariants of its tangent cone coincide when
this last is a Cohen-Macaulay ring. Then, their values are completely determined by the differences of
the minimal number of generators of the consecutive powers of the maximal ideal.

\begin{cor}
\label{b=2}

Let $(A,\fm)$ be a Cohen-Macaulay local ring. If $b=2$ then
\begin{enumerate}
\item $\alpha_i=\beta_i=\gamma_i=1 \textrm{ for } 1\leq i \leq
e-1$,

\item $\displaystyle{G(\fm) \cong K[X]\oplus (K[X](-1))\oplus
\cdots \oplus (K[X](-e+1))}$,

\item $H_{G(\fm/xA)}(z)=1+z+\cdots +z^{e-1}$,

\item In the equicharacteristic complete case

 $ A'/A \cong (K[[X]]/XK[[X]]) \oplus
\cdots \oplus  (K[[X]]/ X^{e-1}K[[X]]).$
\end{enumerate}

\end{cor}
\begin{proof}
It is known that $b=2$ implies that $G(\fm)$ is Cohen-Macaulay and
$\mu(\fm^i)- \mu(\fm^{i-1})=1$ for $i=1,\dots,r$ (see for example
\cite[Proposition 26]{CZ}) and $e=r+1$. So the result is obtained
by applying Proposition \ref{GCM} and Proposition \ref{mainCM}.
\end{proof}

We recall that $e=b+\la (\fm^2/x\fm)$. So, one says that $A$ has
minimal multiplicity when $e=b$ and that $A$ has almost minimal
multiplicity if $b=e+1$.

 When the ring has minimal multiplicity, or equivalently
has reduction number one, the tangent cone is Cohen-Macaulay and
the computation of its invariants, and hence of the
micro-invariants and Apery numbers  of the ring is direct.

\begin{cor}
\label{b=e}

 Let $(A,\fm)$ be a Cohen-Macaulay local ring with minimal
multiplicity, then

\begin{enumerate}
\item $\alpha_1=\beta_1=\gamma_1= e-1$,

\item $\displaystyle{G(\fm) \cong K[X]\oplus (K[X](-1))^{e-1}}$,

\item $H_{G(\fm/xA)}(z)=1+(e-1)z$,

\item In the equicharacteristic complete case

 $ A'/A \cong (K[[X]]/XK[[X]])^{e-1}.$
\end{enumerate}
\end{cor}

We note that Corollary 5.1 (4) and Corollary 5.2 (4) were already shown in
\cite[Proposition 4.1]{E}.

The case of rings with almost minimal multiplicity will provide
examples of micro-invariants and Apery numbers for rings for which
their tangent cones are not Cohen-Macaulay. In this case the
maximal ideal is a "Sally ideal", which means that
$\la(\fm^2/x\fm)=1$. Sally ideals are studied in \cite{R} by M. E.
Rossi, \cite{JPV} by A. V. Jayanthan, T. J. Puthenpurakal and J.
K. Verma and \cite{CZ} by the authors. We collect in a lemma some
known results for  this case.

\begin{lem}
\label{almost} Let $(A,\fm)$ be a Cohen-Macaulay local ring with
almost minimal multiplicity $e$. Then
\begin{enumerate}
\item $\fm^2 $ is not contained in $(x) $.

 \item $\fm^{n+1}\subseteq x\fm^{n-1}$ for $n\geq 2$.

 \item $\la(\fm^{n+1}/x\fm^n)=1$ for $1\leq n\leq r-1$.

 \item
$\mu(\fm^{n})=\begin{cases}\mu (\fm) \textrm{ for } 1\leq n\leq r-1 \\
\mu (\fm)+1 \textrm{ for } n \geq r \end{cases}$

\item $G(\fm)$ is Cohen-Macaulay if and only the reduction number
of $A$ is $2$, if and only if $\mu(\fm^2)=\mu(\fm)+1$.
\end{enumerate}
\end{lem}
\begin{proof}
Observe that $A$ has almost minimal embedding dimension if and
only if $\la(\fm^2/x\fm)=1$.

If $\fm^2\subseteq (x)$ then the exact sequence
$$ 0\longrightarrow \fm^2/x\fm \longrightarrow (x)/x\fm
\longrightarrow (x)/\fm^2 \longrightarrow 0$$ gives, by using the
additivity of the length the equality  $(x)=\fm^2$ which is not
possible since $x$ is part of a minimal set of generators for
$\fm$.

In order to prove that $\fm^3 \subseteq x\fm$ we consider the
exact sequence
$$ 0\longrightarrow (\fm^3+x\fm)/x\fm \longrightarrow \fm^2/x\fm
\longrightarrow \fm^2/(\fm^3+x\fm)\longrightarrow 0, $$ the
Nakayama's Lemma and the additivity of the length gives the
result.

The assertion (3) can be found in the proof of \cite[Corollary
1.7]{R} and (5) in \cite[Theorem 3.3]{JPV}.

The equality $b+1=e=\la(\fm^n/x\fm^n)=\mu (\fm^n)+\la
(\fm^{n+1}/x\fm^n)$ gives the last assertion since $\la
(\fm^{n+1}/x\fm^n)=0$ for $n\geq r$ and $\la( \fm^{n+1}/x\fm^n)=1$
for $n<r$.
\end{proof}

\begin{cor}
\label{almost2} Let $(A,\fm)$ be a Cohen-Macaulay local ring with
almost minimal multiplicity $e$ and reduction number 2, then

\begin{enumerate}
\item $\alpha_1=\beta_1=\gamma_1 =e-2$ and
$\alpha_2=\beta_2=\gamma_2=1$.

\item $\displaystyle{G(\fm) \cong K[X]\oplus (K[X](-1))^{e-2}}
\oplus K[X](-2)$,

\item $H_{G(\fm/xA)}(z)=1+(e-2)z+z^2$,

\item In the equicharacteristic complete case

 $ A'/A \cong (K[[X]]/XK[[X]])^{e-2} \oplus K[[X]]/X^2K[[X]].$
\end{enumerate}
\end{cor}

\begin{cor}
\label{almost3}
 Let $(A,\fm)$ be a Cohen-Macaulay local ring with
almost minimal multiplicity and reduction number 3. Then
$$\begin{array}{ll}
(\alpha_1, \alpha_2,\alpha_3) &=(e-3,1,1)\\ (\beta_1,
\beta_2,\beta_3) &=(e-3,2,0)
\\(\gamma_1,\gamma_2,\gamma_3) &=(e-2,1,0). \end{array}$$
\end{cor}

\begin{proof}
By Lemma \ref{alpha} and Lemma \ref{almost} (4) we get that
$$ \begin{array}{l} \alpha_1=b-1-\la ((\fm^3:x)/\fm^2)\\
\alpha_2= \la ((\fm^3:x)/\fm^2)\\
\alpha_3=1. \end{array}$$

Now, again by Lemma \ref{almost} (2) and (3) we have $\la
((\fm^3:x)/\fm^2)=\la((x\fm \cap \fm^3)/x\fm^2)=
 \la(\fm^3/x\fm^2)=1$ and so the statement for the $\alpha_i's$.

 In order to determine the values of the micro-invariants and the Apery numbers we just need to apply respectively
 Lemma \ref{beta-alpha} and Lemma \ref{alpha-gamma}.
\end{proof}

\begin{cor}
\label{almostr}
 Let $(A,\fm)$ be a Cohen-Macaulay local ring with
almost minimal multiplicity $e$ and reduction number $r\geq 3$.
Then
\begin{enumerate}
\item $(\gamma_1,\dots ,\gamma_r)=(e-2,1,0\dots,0)$.
 \item $\alpha_r=\alpha_{r-1}=1$.
 \item $\beta_r=0$.
\end{enumerate}

\end{cor}
\begin{proof}
By definition, $\gamma_i=\la((\fm^i+xA)/(\fm^{i+1}+xA)$  and,
since $A$ has almost  minimal  multiplicity, $\fm^i \subseteq xA$
for $i\geq 3$, hence $\gamma_i=0$ for $i\geq 3$. Moreover,
$\gamma_1=\mu (\fm) -1=e-2$ and
$\gamma_2=\la((\fm^2+xA)/(\fm^3+xA))=\la((\fm^2+xA)/xA)=\la(\fm^2/x\fm)=1$.
So, (1) is proved.

For (2), combining  Lemma \ref{alpha} and Lemma \ref{almost} (4)
one has that $\alpha_r=1$ and $$\alpha_{r-1}=\la
((\fm^{r-2}\cap(\fm^r:x))/\fm^{r-1})=\la((x\fm^{r-2}\cap
\fm^r)/x\fm^{r-1}).$$ Moreover Lemma \ref{almost} (2) gives the
inclusion $\fm^r \subseteq x\fm^{r-2}$ and so the equalities
$\alpha_{r-1}=\la(\fm^r/x\fm^{r-1})=1$.

Finally, we can obtain (3) by
 Proposition \ref{beta-alpha} which provides, in the almost minimal
 multiplicity case, the equality $\beta_r=\alpha_r -\la((\fm^r:x)/\fm^{r-1})=1-\la(\fm^r/x\fm^{r-1})=0$.
 \end{proof}

\subsection{Numerical semigroups rings}

Let $\mathbb N$ be the set of non-negative integers.  Recall that
a numerical semigroup $S$ is a subset of $\mathbb N$ that is
closed under addition, contains the zero element and has finite
complement in $\mathbb N$. A numerical semigroup $S$ is always
finitely generated; that is, there exist integers $n_1,\dots,n_l$
such that $S=\langle n_1,\dots ,n_l \rangle=\{ \alpha_1 n_1
+\cdots  +\alpha_l n_l ; \alpha_i\in \mathbb N\}$. Moreover, every
numerical semigroup has an unique minimal system of generators
$n_1,\dots ,n_{b(S)}$. The least integer belonging to $S$ is known
as the multiplicity of $S$ and it is denoted by $e(S)$.

A relative ideal of $S$ is a nonempty set $I$ of non-negative
integers such that $I+S\subset I$ and $d+I\subseteq S$ for some
$d\in S$. An ideal of $S$ is then a relative ideal of $S$
contained in $S$. If $i_1,\dots ,i_k$ is a subset of non-negative
integers, then the set $\{i_1,\dots ,i_k\}+S=(i_1+S)\cup \cdots
\cup (i_k+S)$ is a relative ideal of $S$ and $i_1,\dots ,i_k$ is a
system of generators of $I$. Note that, if $I$ is an ideal of $S$,
then $I\cup \{0\}$ is a numerical semigroup and so $I$ is finitely
generated. We denote by $M$ the maximal ideal of $S$, that is,
$M=S\setminus \{0\}$. $M$ is then the ideal generated by a system
of generators of $S$. If $I$ and $J$ are relative ideals of $S$
then $I+J=\{i+j;i\in I, j\in
 J\}$ is also a relative ideal of $S$. Finally, we denote by $\Ap (I)$ the Apery
set of $I$ with respect to $e(S)$, defined as the set of the
smallest elements in $I$ in each residue class module $e(S)$.

Let $V=k[[t]]$ be the formal power series ring over a field $k$.
Given a numerical semigroup $S=\langle n_1,\dots ,n_b \rangle$
minimally generated by $0<e=e(S)=n_1<\cdots < n_b=n_{b(S)}  $ we
consider the ring associated to $S$ defined as
$A=k[[S]]=k[[t^{n_1},\dots ,t^{n_b}]] \subseteq V$. Let
$\fm=(t^{n_1},\dots ,t^{n_b})$ be the maximal ideal of $A$. Then
$A$ is a Cohen-Macaulay local ring of dimension one with
multiplicity $e$ and embedding dimension $b$.  These kind of rings
are known as numerical semigroup rings. The ideals
$(t^{i_1},\dots,t^{i_k})$ of $A$ are such that for $v$, the
$t$-adic valuation, $ v((t^{i_1},\dots,t^{i_k}))=\{i_1,\dots
,i_k\}+S$. In particular, for the ideals $\fm^n$ one has
$v(\fm^n)=nM = M + \stackrel{n}{\cdots} +M$. Note that the element
$t^e$ generates a minimal reduction of $\fm$ and, in terms of
semigroups, $(n+1)M\subseteq nM$ for $n\geq 0$ (we will set
$\fm^0:=A$) and  $(n+1)M=e+nM$ for all $n\geq r$. Also, for these
rings, the first neighborhood ring $A'=k[[S']]$, is a numerical
semigroup ring, with $S'=\langle n_1, n_2-n_1, \dots ,b_b-n_1
\rangle$.

Let $A=k[[S]]$ be a numerical semigroup ring of multiplicity $e$
and reduction number $r$.

If we put
\[ \Ap (nM)=\{w_{n,0},\dots,\w_{n,i}, \dots ,\w_{n,e-1}\}\] for $n \geq
0$, then
\[ \fm^n= Wt^{\w_{n,0}} \oplus \cdots \oplus Wt^{\w_{n,i}} \oplus \cdots \oplus Wt^{\w_{n,e-1}}.\]

The set $\{t^{w_{0,0}},\dots , t^{w_{0,e-1}}\}$ is an Apery basis
of $k[[S]]$ (with respect to $x=t^e$ and also a BF-Apery basis)
and fixed $i$, $1\leq i \leq e-1$ one has that $w_{n+1
,i}=w_{n,i}+\epsilon \cdot e$ where $\epsilon \in \{0,1\}$ and
$w_{n+1 ,i}=w_{n,i}+ e$ for $n\geq r$. These facts are proved in
\cite[Lemma 2.1 and Lemma 2.2]{CZ2}.

 We will show that all the
invariants defined in the previous sections can be computed in
terms of the information contained in the Apery table:

\begin{table}[h]
\renewcommand\arraystretch{1.5}
\[
\begin{array}{|c|c|c|c|c|c|c|}\hline
 \Ap(S)& \w_{0,0} &\w_{0,1}&\cdots  &
 \w_{0,i} &\cdots & \w_{0,e-1} \\
\hline
\Ap(M)& \w_{1,0} &\w_{1,1}&\cdots & \w_{1,i} & \cdots & \w_{1,e-1}\\
\hline \vdots& \vdots & \vdots &\vdots & \vdots & \vdots & \vdots
\\ \hline
\Ap(nM)&\w_{n,0} &\w_{n,1}&\cdots & \w_{n,i} & \cdots &
\w_{n,e-1}\\
\hline
\vdots & \vdots & \vdots &\vdots & \vdots & \vdots & \vdots \\
\hline
\Ap(rM) & \w_{r,0} &\w_{r,1}& \cdots & \w_{r,i} & \cdots  & \w_{r,e-1}\\
\hline  \end{array} \]
 \end{table}

Previously we recall the following notation introduced in
\cite{CZ2}.
\medskip

Let $E = \{w_0, \dots , w_m\}$ be a set of integers. We call it a
stair if $w_0 \leq \cdots \leq w_m$. Given a stair, we say that a
subset $L = \{ w_i, \dots, w_{i+k}\}$ with $k\geq 1$ is a landing
of length $k$ if $w_{i-1} < w_i = \cdots = w_{i+k} < w_{i+k+1}$
(where $w_{-1} = -\infty$ and $w_{m+1} = \infty$). In this case,
the index $i$ is the beginning of the landing: $s(L)$ and the
index $i+k$ is the end of the landing: $e(L)$. A landing $L$ is
said to be a true landing if $s(L) \geq 1$. Given two landings $L$
and $L'$, we set $L < L'$ if $s(L) < s(L')$. Let $l(E)+1$ be the
number of landings and assume that $L_0< \cdots <L_{l(E)}$ is the
set of landings. Then, we define following numbers:

\begin{itemize}

\item[$\cdot$] $s_j(E)= s(L_j)$, $e_j(E) = e(L_j)$, for each
$0\leq j \leq l(E)$;

\medskip

\item[$\cdot$] $c_j(E) = s_j - e_{j-1}$, for each $1 \leq j \leq
l(E)$.

\medskip

\item[$\cdot$] $k_j(E) = e_j - s_{j}$, for each $1 \leq j \leq
l(E)$.

\end{itemize}

With this notation, for any $1\leq i \leq e-1$, consider the
ladder of values $W^i=\{\w_{n,i}\}_{0\leq n\leq r}$, that is, the
columns of the Apery table,  and define the following integers:
\begin{enumerate}
\item $l_i = l(W^i)$; \item $d_i = e_{l_i}(W^i)$; \item
$b_j^i=e_{j-1}(W^i)$ and $c_j^i=c_j(W^i)$, for $1\leq j\leq l_i$.
\end{enumerate}
Then \cite[Theorem 2.3]{CZ2} says

\[G(\fm)\cong F(t^e)\oplus \bigoplus_{i=1}^{e-1}\left(
F(t^e)(-d_i)
 \bigoplus_{j=1}^{l_i} \frac{F(t^e)}{((t^e)^{\ast})^{c_j^i}F(t^e)}
(-b_j^i)\right).\]

Observe that

\begin{enumerate}
\item [(4)] $b_i=e_0(W^i)$.

\item [(5)] $d_i=b_i+(c^i_1+k^i_1)+\cdots +(c^i_{l_i}+k^i_{l_i})$.
\end{enumerate}

Observe also that if $\Ap(S')=\{\w'_0,\dots ,\w'_{e-1}\}$, then
$\w_{0,i}-\w'_i=a_i \cdot e$ for some positive integers and
\[
\begin{array}{lllll}
A'&= W\oplus & Wt^{\w'_1} &\oplus \cdots \oplus & Wt^{\w'_{e-1}} \\
A &= W\oplus & W(t^e)^{a_1}\cdot t^{\w'_1} &\oplus \cdots \oplus
& W(t^e)^{a_{e-1}}\cdot t^{\w'_{e-1}}
\end{array}
 \]
which show thats $\{ a_1, \dots ,a_{e-1}\}$ are the
micro-invariants of $A$. Moreover, from the equality
$\fm^r=(t^e)^rA'$ it is easy to see that
\begin{enumerate}
\item [(6)] $d_i=a_i+(c^i_1+\cdots + c^i_{l_i})$

\item [(7)] $a_i=b_i+(k^i_1+\cdots + k^i_{l_i})$.
\end{enumerate}

Hence, the Cohen-Macaulay property of the tangent cone is
equivalent to the no existence of true landings in the columns of
the Apery table. Also, each true landing gives a torsion cyclic
submodule of the tangent cone and its beginning and ending
determine the degree and the order of the corresponent torsion
submodule.

Note also that we can read the Hilbert function $H^0(n)=\mu
(\fm^n)$ in the Apery table as the number of steps between the nth
row and the (n+1)th row.
\bigskip

Suppose that $e$, the multiplicity of $S$ (equivalently the
multiplicity of $k[[S]]$), is given. We recall that then, the
embedding dimension $b$ and the reduction number $r$ satisfy $b\leq e$
and $r\leq e-1$. We will show that, in general, the couple $(e,b)$
does not determine the Apery table of $S$. However, in the
extremal cases $(e,2)$ and $(e,e)$ the Apery table is completely
determined.

\begin{ex} Suppose that $S$ has multiplicity $e$.

\begin{itemize}
\item For $b=2$,  we consider  $\{w_1,\dots,\w_{e-1}\}$, with
$\w_1< \cdots < \w_{e-1}$ a suitable permutation of
$\{w_{0,0},\dots,\w_{0,e-1}\}$ the apery set of $S$ (with this
notation $S=<e,\w_1>$). In this case the reduction number is $e-1$
and the Apery table is a square box:

\begin{table}[h]
\renewcommand\arraystretch{1.5}
\[
\begin{array}{|c|c|c|c|c|c|}\hline

  0 &\w_1&\cdots  & \w_i &\cdots & \w_{e-1} \\\hline

 e &\w_1&\cdots & \w_i & \cdots & \w_{e-1}\\\hline

 \vdots&  \vdots &\vdots & \vdots & \vdots & \vdots\\ \hline

ie &\w_1 +(i-1)e&\cdots & \w_i & \cdots &\w_{e-1}\\\hline

\vdots & \vdots & \vdots &\vdots & \vdots & \vdots  \\\hline

 re &\w_1+(r-1)e& \cdots & \w_i+(r-i)e & \cdots  & \w_{e-1}\\
\hline  \end{array} \]
 \end{table}

So, for $1\leq i \leq e-1$ and observing the columns of the table
we have that $a_i=b_i=d_i=i$ and consequently $\alpha_i =\beta_i
=\gamma _i=1$ for $1\leq i \leq r$ as we proved in Corollary
\ref{b=2}. Moreover $\rho=e(e-1)/2$.
\newpage

\item For $b=e$ the reduction number $r$ is equal to $1$, $S$ is
minimally generated by the Apery set $\{w_{0,0},\dots,\w_{0,e-1}\}
$ and the Apery table has two rows:
\begin{table}[h]
\renewcommand\arraystretch{1.5}
\[
\begin{array}{|c|c|c|c|c|c|}\hline
  0 &\w_{0,1}&\cdots  &
 \w_{0,i} &\cdots & \w_{0,e-1} \\
\hline
 e &\w_{0,1}&\cdots & \w_{0,i} & \cdots & \w_{0,e-1}\\
\hline
 \end{array}
\]
\end{table}

So, $a_i=b_i=d_i=1$ for $1\leq i \leq e-1$ and
$\alpha_1=\beta_1=\gamma_1=e-1$ recovering Corollary \ref{b=2} for
numerical semigroup rings. In this case $\rho =e-1$.
\end{itemize}
\end{ex}
\medskip

For $3\leq b\leq e-1$ there are several possibilities  for
the reduction number and the Apery table as shown by the following examples
for $e=5$.

The GAP - Groups, Algorithms, Programming - is a system for
Computational Discrete Algebra \cite{GAP4}. On the basis of GAP,
Manuel Delgado, Pedro A. Garcia-S\'{a}nchez and  Jos\'{e} Morais
have developed the NumericalSgps package \cite{NumericalSgps}. Its
aim is to make available a computational tool to deal with
numerical semigroups. We can determine the values of the diverse
families of invariants if we know the Apery sets of the sum ideals
$nM$, where $M$ is the maximal ideal of $S$. On the other hand,
from its definition we have that the Apery set of $nM$ can be
calculated as $ \Ap (nM)=nM \setminus ((e+S)+ nM)$, a computation
that can be performed by using the NumericalSgps package. The
following examples are just a sample of these computations.

\begin{ex} We assume in this example that $(e, b)=(5,3)$.

\begin{itemize}
\item Set $S=<5,6,7>$. The reduction number is $2$ and the Apery
table is in this case
\begin{table}[h]
\label{5,6,7}
\renewcommand\arraystretch{1.5}
\[
\begin{array}{|c|c|c|c|c|}\hline  0&6&7&13&14\\ \hline
5&6&7&13&14\\ \hline 10&11&12&13&14\\ \hline
  \end{array},
 \]
 \end{table}

\noindent so, $a_i=b_i=d_i$ for $1\leq i \leq 4$, $b_1=b_2=1$ and
$b_3=b_4=2$. Also $(\alpha_1,\alpha_2)=(\beta_1,\beta_2)=(\gamma_1,\gamma_2)=(2,2)$ and $\rho=6$.

 \item Set $S=<5,6,9>$. The reduction number is $3$ and
the Apery table in this case is

\begin{table}[h]
\renewcommand\arraystretch{1.5}
\[
\begin{array}{|c|c|c|c|c|}\hline  0&6&12&18&9\\ \hline
5&6&12&18&9\\ \hline 10&11&12&18&14\\ \hline 15&16&17&18&19\\
\hline
  \end{array},
 \]
 \end{table}

\noindent so, $a_i=b_i=d_i$, $b_1=b_4=1$, $b_2=2$ and $b_3=3$.
Also $(\alpha_1,\alpha_2, \alpha_3)=(\beta_1,\beta_2, \beta_3)=(\gamma_1,\gamma_2, \gamma_3)=(2,1,1)$ and $\rho=7$.

\item Set $S_1=<5,6,13>$, $S_2=<5,6,14>$ and $S_3=<5,6,19>$. The
reduction number in these cases is $4$.

The Apery table for $S_1$ is

\begin{table}[h]
\renewcommand\arraystretch{1.5}
\[
\begin{array}{|c|c|c|c|c|}\hline  0&6&12&13&19\\ \hline
5&6&12&13&19\\ \hline 10&11&12&18&19\\ \hline 15&16&17&18&24\\
\hline 20&21&22&23&24\\ \hline
  \end{array}, \]
\end{table}

\noindent and so, the invariants are
\[
\begin{array}{l}
(\alpha_1,\alpha_2,\alpha_3,\alpha_4)=(1,1,1,1), \alpha_{1,1}=1,
\alpha_{2,1}=1,\\
 (\beta_1,\beta_2,\beta_3,\beta_4)=(1,2,1,0),\\
(\gamma_1,\gamma_2,\gamma_3,\gamma_4)=(2,2,0,0).\end{array}\]

The Apery table for $S_2$ is

\begin{table}[h]
\renewcommand\arraystretch{1.5}
\[
\begin{array}{|c|c|c|c|c|}\hline  0&6&12&18&14\\ \hline
5&6&12&18&14\\ \hline 10&11&12&18&19\\ \hline 15&16&17&18&24\\
\hline 20&21&22&23&24\\ \hline
  \end{array}\]
  \end{table}

\noindent and their invariants are
\[
\begin{array}{l}
(\alpha_1,\alpha_2,\alpha_3,\alpha_4)=(1,1,1,1), \alpha_{1,2}=1\\
(\beta_1,\beta_2,\beta_3,\beta_4)=(1,2,1,0),\\
(\gamma_1,\gamma_2,\gamma_3,\gamma_4)=(2,1,1,0).\end{array}\]

Finally,  for $S_3$ the Apery table is
\begin{table}[h]
\renewcommand\arraystretch{1.5}
\[
\begin{array}{|c|c|c|c|c|}\hline  0&6&12&18&19\\ \hline
5&6&12&18&19\\ \hline 10&11&12&18&24\\ \hline 15&16&17&18&24\\
\hline 20&21&22&23&24\\ \hline
  \end{array}\]
 \end{table}

\noindent which produces the invariants
\[
\begin{array}{l}
(\alpha_1,\alpha_2,\alpha_3,\alpha_4)=(1,1,1,1), \alpha_{1,1}=1\\
(\beta_1,\beta_2,\beta_3,\beta_4)=(1,1,2,0),\\
(\gamma_1,\gamma_2,\gamma_3,\gamma_4)=(2,1,1,0).\end{array}
 \]

 \end{itemize}
Observe that $(e,b,r)$ neither determines the Apery table nor any
of the families of invariants.
\end{ex}

\begin{ex} Suppose that $(e,b)=(5,4)$
\begin{itemize}
\item Set $S=<5,6,7,8>$. In this case $r=2$, and analyzing its
Apery table
\begin{table}[h]
\renewcommand\arraystretch{1.5}
\[
\begin{array}{|c|c|c|c|c|}\hline  0&6&7&8&14\\ \hline
5&6&7&8&14\\ \hline 10&11&12&13&14\\ \hline
  \end{array},
 \]
 \end{table}

\noindent we obtain $(\alpha_1,\alpha_2)=(\beta_1,\beta_2)=(\gamma_1,\gamma_2)=(3,1)$.

\item Set $S=<5,6,9,13>$. In this case  $r=3$, the Apery table is

\begin{table}[h]
\renewcommand\arraystretch{1.5}
\[
\begin{array}{|c|c|c|c|c|}\hline  0&6&12&13&9\\ \hline
5&6&12&13&9\\ \hline 10&11&12&18&14\\ \hline 15&16&17&18&19\\
\hline
  \end{array},
 \]
 \end{table}

\noindent and
\[ \begin{array}{l}
(\alpha_1,\alpha_2,\alpha_3)=(2,1,1), \alpha_{1,1}=1,\\
(\beta_1,\beta_2,\beta_3)=(2,2,0),\\
(\gamma_1,\gamma_2,\gamma_3)=(3,1,0). \end{array}\]

 \item Set $S=<5,6,13,14>$. In this case $r=4$ and the Apery table

\begin{table}[h]
\renewcommand\arraystretch{1.5}
\[
\begin{array}{|c|c|c|c|c|}\hline  0&6&12&13&14\\ \hline
5&6&12&13&14\\ \hline 10&11&12&18&19\\ \hline 15&16&17&18&24\\
\hline 20&21&22&23&24\\ \hline
  \end{array}, \]
\end{table}

\noindent gives
\[\begin{array}{l}
(\alpha_1,\alpha_2,\alpha_3,\alpha_4)=(1,1,1,1),\alpha_{1,1}=1,\alpha_{1,2}=1,\\
(\beta_1,\beta_2,\beta_3,\beta_4)=(1,3,0,0),\\
(\gamma_1,\gamma_2,\gamma_3,\gamma_4)=(3,1,0,0).\end{array}\]

\end{itemize}
\end{ex}

\bibliographystyle{amsalpha}

\end{document}